\documentclass[12pt,a4paper,reqno]{amsart} 
\usepackage{amsmath,amsthm,enumerate,graphicx,amssymb,times,latexsym}
\usepackage[mathscr]{euscript}
\usepackage{latexsym}
\usepackage{amsmath}
\usepackage{mathrsfs}
\usepackage{calc}
\usepackage{cite}

\newcommand{\rr}[1]{\mathbf R^{#1}}

\newcommand{\vrum}{\vspace{0.1cm}}

\newcommand{\masfR}{\mathsf R}

\def\be{\begin{equation}}
\def\ee{\end{equation}}
\def\bena{\begin{eqnarray*}}
\def\ena{\end{eqnarray*}}

\def\mR{\mathbb{R}}   \def\mC{\mathbb{C}}

\newcommand{\aaf}{A_a^{\varphi}}

\numberwithin{equation}{section}

\newtheorem{thm}{Theorem}
\numberwithin{thm}{section}
\newtheorem{prop}[thm]{Proposition}
\newtheorem{cor}[thm]{Corollary}
\newtheorem{lemma}[thm]{Lemma}
\theoremstyle{definition}
\newtheorem{defn}[thm]{Definition}
\theoremstyle{remark}
\newtheorem{rem}[thm]{Remark}

\begin{document}

\begin{abstract}
We introduce multilinear  localization operators in terms of the short-time Fourier transform, and  multilinear
Weyl pseudodifferential operators. We prove that such localization operators are in fact
Weyl pseudodifferential operators whose symbols are given by the convolution between the symbol of the localization operator and the multilinear Wigner transform. For such interpretation we use the kenrel theorem for the Gelfand-Shilov space. Furthermore, we
study the continuity properties of the multilinear  localization operators on modulation spaces. Our results extend some known results when restricted to the linear case.
\end{abstract}

\title{Continuity properties of multilinear  localization operators on modulation spaces}

\author{Nenad Teofanov}
\address{Department of Mathematics and Informatics,
University of Novi Sad, Novi Sad, Serbia}
\email{nenad.teofanov@dmi.uns.ac.rs}
\thanks{
This research is supported by MPNTR of Serbia, projects no. 174024, and DS 028 (TIFMOFUS).}


\maketitle

\par

\section{Introduction}\label{sec0}

Multilinear localization operators were first introduced in \cite{CKasso}
and their continuity properties are formulated in terms of modulation spaces. The
key point is the interpretation of these operators as multilinear Kohn-Nirenberg pseudodifferential operators.
The  multilinear pseudodifferential operators were already studied in the context of modulation spaces in \cite{BGHKasso},
see also a more recent contribution \cite{MOPf} where such approach is strengthened and applied to the bilinear and trilinear Hilbert transform.

\par

Our approach is related to  Weyl pseudodifferential operators instead,
with another (Weyl) correspondence between the operator and its symbol. Both correspondences are particular cases of the
so-called $\tau-$pseudodifferential operators, $\tau \in [0,1]$. For $ \tau = 1/2$ we obtain Weyl operators, while for $\tau = 0 $ we recapture Kohn-Nirenberg operators. We refer to \cite{CDT, CNT} for the recent contribution in that context (see also the references given there).

\par

The Weyl correspondence provides an elegant interpretation of localization operators as Weyl  pseudodifferential operators. This is given by the formula that contains the Wigner transform which is, together with the short-time Fourier transform, the main tool in our investigations. We refer to \cite{Wong1998,Gosson} for more details on the Wigner transform.

\par

In signal analysis, different localization techniques are used to describe signals which are as concentrated as possible in general regions of the phase space. This motivated I. Daubechies to address these questions by introducing certain localization operators in the  pioneering  contribution \cite{Daube88}. Afterwards, Cordero and Grochenig made an essential contribution in the context of time-frequency analysis, \cite{CorGro2003}. Among other things, their results emphasized the role played by modulation spaces in the study of localization operators.

\par

In this paper we first recall the basic facts on  modulation spaces  in Section \ref{sec01}. Then, in Section \ref{sec02},
following the definition of bilinear localization operators given in \cite{Teof2018}
we introduce multilinear localization operators, Definition \ref{locopdef}. Then we define
the multilinear Weyl pseudodifferential operators and give their weak formulation in terms of the
multilinear Wigner transform (Lemma \ref{L-and-W}). By using the kernel theorem for Gelfand-Shilov spaces,  Theorem \ref{kernelteorema}, we prove that the multilinear localization operators can be interpreted as multilinear Weyl
pseudodifferential operators in the same way as in the linear case, Theorem \ref{Weyl connection lemma}.

In Section \ref{sec03} we first recall two results from \cite{CN}: the (multilinear version of) sharp integral bounds for the Wigner transform, Theorem \ref{thm-cross-Wigner}, and continuity properties of pseudodifferential operators on modulation spaces,
Theorem \ref{bounded Weyl-extended}. These results, combined with the convolution estimates for modulation spaces from \cite{TJPT2014}, Theorem \ref{mainconvolution}, are then used to prove the main result of the continuity properties of
multilinear localization operators on modulation spaces, Theorem \ref{conv-a-cross-Wigner}.

\noindent \textbf{Notation.}
The Schwartz space of rapidly decreasing smooth functions is denoted by  $\mathcal{S} (\mathbb{R}^d)$, and its dual space of tempered
distributions is denoted by  $\mathcal{S}'(\mathbb{R}^d)$. We use the brackets
$\langle f,g\rangle$ to denote the extension of the inner product
$\langle f,g\rangle=\int f(t){\overline {g(t)}}dt$ on $L^2(\mathbb{R}^d)$
to any pair of dual spaces.
The Fourier transform is normalized to be ${\hat   {f}}(\omega)=\mathcal{F} f(\omega)
=\int f(t)e^{-2\pi i t\omega}dt$. The involution $ f^\ast $ is $ f^\ast (\cdot) = \overline{f (-\cdot)},$
and the convolution of $f$ and $g$ is given by
$ f * g (x) = \int f(x-y) g (y) dy,$ when the integral exists.

\par

We denote by $\langle \cdot \rangle ^s$  the polynomial weights
$$
 \langle (x, \omega) \rangle^s=  (1+|x|^2+|\omega|^2)^{s/2},\quad
   (x,\omega)\in \mathbb{R}^{2d}, \,\quad s\in \mathbb{R},\,
$$
and $\langle x \rangle = \langle 1 + |x|^2\rangle ^{1/2}, $ when $ x \in \mathbb{R}^d.$

We use the notation $A\lesssim B$ to indicate that $A\leq c B$ for a suitable constant $c>0$,
whereas $A \asymp B$ means that $c^{-1} A \leq B \leq c A$ for some $c\geq 1$.

\par

\noindent \textbf{The Gelfand-Shilov space and Weyl pseudodifferential operators.}
The Gelfand-Shilov type space of analytic functions   $  {\mathcal S} ^{(1)} (\mathbb{R}^d ) $ is
given by
$$ f \in  {\mathcal S}^{( 1)} (\mathbb{R}^d) \Longleftrightarrow
\| f(x)  e^{h\cdot |x|}\|_{L^\infty} < \infty \;
\; \text{and} \;
\| \hat f (\omega)  e^{h\cdot |\omega|}\|_{L^\infty}< \infty, \;\; \forall  h > 0.
$$

Any $ f \in  {\mathcal S}^{( 1)} (\mathbb{R}^d) $ can be extended to a
holomorphic function $f(x+iy)$ in the strip
$ \{ x+iy \in \mC ^d \; : \; |y| < T \} $ some $ T>0$, \cite{GS, NR}.
The dual space of  $  {\mathcal S} ^{(1)} (\mathbb{R}^d ) $ will be denoted by
$ {\mathcal S} ^{(1)'}   (\mathbb{R}^d ). $

\par

The space $  {\mathcal S} ^{(1)} (\mathbb{R}^d ) $ is nuclear, and we will use the following kernel theorem in the context of $ {\mathcal S}^{( 1)} (\mathbb{R}^d) $.

\begin{thm} \label{kernelteorema}
Let $ \mathcal{L}_b (\mathcal{A}, \mathcal{B}) $ denote the continuous linear mapping between the spaces
$\mathcal{A}$ and $ \mathcal{B} $.
Then the following isomorphisms hold:
\begin{itemize}
\item[1)]$ \displaystyle {\mathcal S}^{(1)} (\mathbb{R}^{d_1}) \hat{\otimes} {\mathcal S}^{(1)} (\mathbb{R}^{d_2}) \cong
{\mathcal S}^{(1)} (\mathbb{R}^{d_1+d_2}) \cong
\mathcal{L}_b ( {\mathcal S}^{(1)'} (\mathbb{R}^{d_1}), {\mathcal S}^{(1)} (\mathbb{R}^{d_2})),
$
\item[2)]
$ \displaystyle
{\mathcal S}^{(1)'} (\mathbb{R}^{d_1}) \hat{\otimes} {\mathcal S}^{(1)'} (\mathbb{R}^{d_2}) \cong
{\mathcal S}^{(1)'} (\mathbb{R}^{d_1+d_2}) \cong
\mathcal{L}_b ( {\mathcal S}^{(1)} (\mathbb{R}^{d_1}), {\mathcal S}^{(1)'} (\mathbb{R}^{d_2})).
$
\end{itemize}
\end{thm}

Theorem \ref{kernelteorema} is a special case of \cite[Theorem 2.5]{Teof2015}, see also \cite{Prang},
so we omit the proof. We refer to the classical reference \cite{Trev} for kernel theorems and nuclear spaces,
and in particular to Theorem 51.6 and its Corollary related to $  {\mathcal S} (\mathbb{R}^d ) $ and
$  {\mathcal S}' (\mathbb{R}^d ) $, which will be used later on.

\par

The isomorphisms in Theorem \ref{kernelteorema} 2)  imply that for a given kernel-distribution
$ k(x,y) $ on $\mathbb{R}^{d_1+d_2}$ we may associate a continuous linear mapping
$k$ of $  {\mathcal S}^{(1)} (\mathbb{R}^{d_2})$ into
$ {\mathcal S}^{(1)'} (\mathbb{R}^{d_1})$ as follows:
$$
\langle k_\varphi, \phi \rangle = \langle k(x,y), \phi(x) \varphi(y) \rangle, \;\;\; \phi \in
 {\mathcal S}^{(1)} (\mathbb{R}^{d_1}),
$$
which is commonly written as $k_\varphi (\cdot) = \int k(\cdot, y )\varphi(y) dy.$
The correspondence between $k(x,y) $ and $k$ is an isomorphism and this fact will be used
in the proof of Theorem \ref{Weyl connection lemma}.

\par

Let  $\sigma \in {\mathcal S}^{( 1)} (\mathbb{R}^{2d})$. Then the
Weyl pseudodifferential operator $ L_\sigma$ with the Weyl symbol $\sigma$
can be defined as the oscillatory integral:
$$
L_\sigma f (x) = \iint \sigma (\frac{x+y}{2}, \omega) f(y) e^{2\pi i (x-y)\cdot \omega} dy d\omega, \;\;\;
f \in {\mathcal S}^{( 1)} (\mathbb{R}^{d}).
$$
This definition extends to each $\sigma \in {\mathcal S}^{( 1)'} (\mathbb{R}^{2d})$, so that
$L_\sigma$ is a continuous mapping from $ {\mathcal S}^{( 1)} (\mathbb{R}^{2d}) $ to
$ {\mathcal S}^{( 1)'} (\mathbb{R}^{2d})$.
If
\begin{equation} \label{cross-Wigner}
W(f,g)(x,\omega)= \int f(x+\frac{t}2)\overline{g(x-\frac{t}2)} e^{-2\pi     i\omega t}\,dt,  \quad\quad f,g\in {\mathcal S}^{( 1)} (\mathbb{R}^d),
\end{equation}
denotes the  Wigner transform, also known as the cross-Wigner distribution,
then the following formula holds:
$$
\langle L_\sigma f,g \rangle = \langle \sigma, W(g,f) \rangle,
\quad\quad f,g\in {\mathcal S}^{( 1)} (\mathbb{R}^d),
$$
for each  $\sigma \in {\mathcal S}^{( 1)'} (\mathbb{R}^{2d})$,  see e.g. \cite{folland89, book, Wong1998}.

\section{Modulation Spaces} \label{sec01}

In this section we collect some facts on modulation spaces which will be used in Section \ref{sec03}.
First we introduce the short-time Fourier transform in the context of duality between the Gelfand--Shilov space
$ {\mathcal S}^{( 1)} (\mathbb{R}^d) $ and its dual space of tempered ultra-distributions
$  {\mathcal S}^{( 1)'} (\mathbb{R}^{2d})$ as follows.

\par

The short-time Fourier transform (STFT in the sequel)
of $f \in  {\mathcal S} ^{(1)} (\mathbb{R}^d ) $ with respect to the window
$g  \in {\mathcal S} ^{(1)} ( \mathbb{R}^d) \setminus 0 $
is defined by
 \begin{equation}
   \label{STFT!}
   V_g f(x,\omega)= \langle f,M_\omega T_x g \rangle =
   \int_{\mR ^d}
 f(t)\, {\overline {g(t-x)}} \, e^{-2\pi i\omega t}\,dt,
 \end{equation}
where the
translation operator  $T_x$  and the modulation operator $ M_{\omega }$  are given by
\begin{equation}   \label{trans-mod}
 T_x f(\cdot)=f(\cdot-x)\quad{\rm and}\quad M_{\omega }f(\cdot)= e^{2\pi i \omega \cdot}f(\cdot) \;\;\; x,\omega \in \mathbb{R}^d.
\end{equation}
\par

The map $(f,g )\mapsto V_g f$ from
$\mathcal{S} ^{(1)}(\mathbb{R}^d)\otimes \mathcal{S} ^{(1)}(\mathbb{R}^d)$ to
$\mathcal{S} ^{(1)}(\mathbb{R}^{2d})$
extends uniquely to a continuous operator from
$\mathcal{S} ^{(1)'} (\mathbb{R}^d)\otimes \mathcal{S} ^{(1)'}(\mathbb{R}^d)$
to $\mathcal{S} ^{(1)'} (\mathbb{R}^{2d})$ by duality.

\par

Moreover, for a fixed $g \in {\mathcal S}^{(1)} (\mathbb{R}^d )\setminus 0$
the following characterization holds:
$$
f \in  {\mathcal S}^{(1)} (\mathbb{R}^d ) \quad  \Longleftrightarrow \quad
 V_g f  \in {\mathcal S}^{(1)} (\mathbb{R}^{2d} ).
$$

\par

We recall the notation from \cite{Teof2018} related to the bilinear case. For given $ \varphi _1, \varphi _2, f_1, f_2 \in {\mathcal S} ^{(1)} (\mathbb{R}^d )$
we put
\begin{multline} \label{STFTtensor}
V_{\varphi _1 \otimes \varphi _2} (f_1 \otimes f_2) (x,\omega)
= \int_{\mathbb{R}^{2d}} f_1 (t_1) f_2 (t_2) \overline{M_{\omega_1} T_{x_1} \varphi _1 (t_1) M_{\omega_2} T_{x_2} \varphi _2 (t_2)} dt_1 dt_2
\\
= \int_{\mathbb{R}^{2d}} (f_1 \otimes f_2) (t) \overline{(M_{\omega_1} T_{x_1} \varphi _1 \otimes M_{\omega_2} T_{x_2} \varphi _2 )(t)} dt,
\end{multline}
where $  x = (x_1,x_2),$ $ \omega = (\omega_1 ,\omega_2 ),$ $ t = (t_1,t_2) , $
$x_1,x_2,  \omega_1 ,\omega_2, t_1,t_2  \in \mathbb{R}^d.$

To give an interpretation of multilinear operators in the weak sense we note that, when
$ \vec{f} = (f_1,f_2, \dots, f_n) $ and $ \vec{\varphi} = (\varphi_1, \varphi_2, \dots, \varphi_n) $,
$ f_j, \varphi _j   \in {\mathcal S} ^{(1)} (\mathbb{R}^d )$, $ j = 1,2,\dots, n,$
the equation \eqref{STFTtensor}
becomes
\begin{equation} \label{STFTtensor-n}
V_{\vec{\varphi}} \vec{f} (x,\omega)
= \int_{\mathbb{R}^{nd}} \vec{f} (t) \prod_{j=1} ^n \overline{ M_{\omega_j} T_{x_j} \varphi _j (t_j)} dt,
\end{equation}
see also \eqref{product} for the notation.

\par

%

We refer to \cite{GZ01, T2, Teof2015, Teof2016, Toft2012} for  more details on STFT  in other spaces of
Gelfand-Shilov type.
Since we restrict ourselves to weighted
modulation spaces  with polynomial weights in this paper,
we proceed by using
the duality between  $  {\mathcal S} $ and  $  {\mathcal S}' $  instead of the more general
duality between    $  {\mathcal S} ^{(1)} $ and $  {\mathcal S} ^{(1)'} $.
Related results in the framework
of subexponential and superexponential weights can be found in e.g. \cite{CPRT1, CPRT2,Teof2015,Toft2012},
and leave the study of multilinear localization operators in that case for a separate contribution.

\par

Modulation spaces \cite{F1, book} are defined through decay and integrability conditions on STFT,
which makes them suitable for time-frequency analysis, and for the study of localization operators
in particular. They are defined in terms of weighted mixed-norm Lebesgue spaces.

In general, a weight $ w(\cdot) $ on $ \mathbb{R}^{d}$ is a non-negative and continuous function.
The weighted Lebesgue space  $ L^p _w (\mathbb{R}^d)$, $ p \in [1,\infty] $, is the Banach space with the norm
$$
\| f \|_{L^p _w } = \| f w \|_{L^p} = \left ( \int |f(x)|^p w(x) ^p dx \right )^{1/p},
$$
and with the usual modification when $ p=\infty$. When $ w(x) =
\langle  x \rangle ^t $, $t\in \mathbb{R},$ we use the notation $ L^p _t (\mathbb{R}^d) $ instead.

Similarly,
the weighted  mixed-norm space $ L^{p,q} _w (\mathbb{R}^{2d})$, $ p,q \in [1,\infty] $,
consists of (Lebesgue) measurable functions on $ \mathbb{R}^{2d}$ such that
$$
\| F \|_{ L^{p,q} _w} = \left ( \int_{ \mathbb{R}^{d}} \left ( \int_{ \mathbb{R}^{d}}
| F(x,\omega)|^p w(x,\omega) ^p dx \right )^{q/p} d\omega \right )^{1/q} < \infty.
$$
where $ w(x,\omega ) $ is a weight on $ \mathbb{R}^{2d}$.

In particular, when $ w(x,\omega) =
\langle  x \rangle ^t \langle \omega \rangle ^s, $ $s,t\in \mathbb{R},$ we use the notation
$ L^{p,q} _w (\mathbb{R}^{2d}) $ $= L^{p,q} _{s,t} (\mathbb{R}^{2d})$.

Now, modulation space $M^{p,q}_{s,t}(\mathbb{R}^d)$  consists of distributions whose STFT belong to
$ L^{p,q} _{s,t} (\mathbb{R}^{2d})$:

\begin{defn} \label{modspaces}
Let $\phi \in \mathcal{S}(\mathbb{R}^d)  \setminus 0$, $s,t\in \mathbb{R}$, and $p,q\in [1,\infty]$. The
\emph{modulation space} $M^{p,q}_{s,t}(\mathbb{R}^d)$ consists of all
$f\in \mathcal{S}'(\mathbb{R}^d)$ such that
$$
\| f \|_{M^{p,q}_{s,t}} \equiv \left  ( \int _{\mathbb{R}^d} \left ( \int _{\mathbb{R}^d}
|V_\phi f(x,\omega )\langle  x \rangle ^t
\langle \omega \rangle ^s|^p\, dx  \right )^{q/p}d\omega  \right )^{1/q}<\infty
$$
(with obvious interpretation of the integrals when $p=\infty$ or $q=\infty$).
\end{defn}

\par

In special cases we use the usual abbreviations: $M^{p,p} _{0,0} = M^{p},$
$M^{p,p} _{t,t} = M^{p} _t,$  etc.

For the consistency, and according to \eqref{STFTtensor-n}, we denote by $ \mathcal{M}^{p,q}_{s,t} (\mathbb{R}^{nd}) $ the set of
$ \vec{f} = (f_1,f_2, \dots, f_n), $ $ f_j \in {\mathcal S}' (\mathbb{R}^{d})$,  $ j = 1,2,\dots,n$, such that
\begin{equation} \label{modnormtensor}
\| \vec{f} \|_{\mathcal{M}^{p,q}_{s,t}} \equiv \left  ( \int _{\mathbb{R}^{2d}} \left ( \int _{\mathbb{R}^{2d}}
| V_{\vec{\varphi}} \vec{f} (x,\omega)
\langle  x \rangle ^t
\langle \omega \rangle ^s|^p\, dx  \right )^{q/p}d\omega  \right )^{1/q}<\infty,
\end{equation}
where $ \vec{\varphi} = (\varphi_1, \varphi_2, \dots, \varphi_n) $,
$  \varphi _j   \in {\mathcal S}  (\mathbb{R}^d ) \setminus 0$, $ j = 1,2,\dots, n,$ is a given window function.

\par

The kernel theorem for $  {\mathcal S} (\mathbb{R}^{d})$ and $ {\mathcal S}' (\mathbb{R}^{d})$
(see \cite{Trev}) implies that there is an isomorphism between $ \mathcal{M}^{p,q}_{s,t} (\mathbb{R}^{nd}) $
and $M^{p,q}_{s,t}(\mathbb{R}^{nd})$ (which commutes with the operators from \eqref{trans-mod}). This allows us to identify
$ \vec{f} \in  \mathcal{M}^{p,q}_{s,t} (\mathbb{R}^{nd}) $ with (its isomorphic image) $ F \in M^{p,q}_{s,t}(\mathbb{R}^{nd})$
(and vice versa). We will use this identification whenever convenient and without further mentioning.

\begin{rem}
The original definition of modulation spaces given in \cite{F1} deals with more general
{\em submultiplicative} weights. We restrict ourselves to the weights of the form $ w(x,\omega)$ $=
\langle  x \rangle ^t \langle \omega \rangle ^s, $ $s,t\in \mathbb{R},$
since the convolution and multiplication estimates which will be used later on
are formulated in terms of weighted spaces with such polynomial weights.
As already mentioned, weights of  exponential type growth are used
in the study of Gelfand-Shilov spaces and their duals in
cf. \cite{GZ01, T2, CPRT2, Toft2012}. We refer to \cite{Gro07} for a
survey on the most important types of weights commonly used in time-frequency analysis.
\end{rem}

\par

The following theorem lists some basic properties of modulation spaces.
We refer to \cite{F1, book} for the proof.

\begin{thm} \label{modproerties}
Let $p,q,p_j,q_j\in [1,\infty ]$ and $s,t,s_j,t_j\in \mathbb{R}$, $j=1,2$. Then:
\begin{enumerate}
\item[{\rm{1)}}] $M^{p,q}_{s,t}(\mathbb{R}^d)$ are Banach spaces, independent of the choice of
$\phi \in \mathcal{S}(\mathbb{R}^d) \setminus 0$;

\item[{\rm{2)}}] if  $p_1\le p_2$, $q_1\le q_2$, $s_2\le s_1$ and
$t_2\le t_1$, then
$$
\mathcal{S}(\mathbb{R}^d)\subseteq M^{p_1,q_1}_{s_1,t_1}(\mathbb{R}^d)
\subseteq M^{p_2,q_2}_{s_2,t_2}(\mathbb{R}^d)\subseteq
\mathcal{ S}'(\mathbb{R}^d);
$$

\item[{\rm{3)}}] $ \displaystyle
\cap _{s,t} M^{p,q}_{s,t}(\mathbb{R}^d)=\mathcal{ S}(\mathbb{R}^d),
\quad
\cup _{s,t}M^{p,q}_{s,t}(\mathbb{R}^d)=\mathcal{ S}'(\mathbb{R}^d); $

\item[{\rm{4)}}] For   $p,q\in [1,\infty )$, the dual of $ M^{p,q}_{s,t}(\mathbb{R}^d)$ is
$ M^{p',q'}_{-s,-t}(\mathbb{R}^d),$ where $ \frac{1}{p} +  \frac{1}{p'} $ $ =
 \frac{1}{q} +  \frac{1}{q'} $ $ =1.$
\end{enumerate}
\end{thm}

Modulation spaces include the following well-know function spaces:
\begin{itemize}
\item[a)] $ M^2 (\mathbb{R}^d) = L^2 (\mathbb{R}^d),$  and $ M^2 _{t,0}(\mathbb{R}^d) = L^2 _t (\mathbb{R}^d);$
\item[b)] The Feichtinger algebra: $ M^1 (\mathbb{R}^d) = S_0 (\mathbb{R}^d);$
\item[c)] Sobolev spaces: $ M^2 _{0,s}(\mathbb{R}^d) = H^2 _s (\mathbb{R}^d) = \{ f \, | \,
\hat f (\omega) \langle \omega \rangle ^s \in  L^2 (\mathbb{R}^d)\};$
\item[d)] Shubin spaces: $ M^2 _{s}(\mathbb{R}^d) = L^2 _s (\mathbb{R}^d) \cap H^2 _s (\mathbb{R}^d) = Q_s (\mathbb{R}^d),$
cf. \cite{Shubin91}.
\end{itemize}

\par

To deal with duality when $ p q = \infty $  we observe that, by a slight modification of \cite[Lemma 2.2]{BGHKasso}
the following is true.

\begin{lemma} \label{dualnost}
Let $ L^0  (\mathbb{R}^{2nd}) $ denote the space of bounded, measurable functions on $ \mathbb{R}^{2nd}$ which vanish at infinity and put
\begin{eqnarray*}
\mathcal{M} ^{0,q} (\mathbb{R}^{nd}) & = & \{ \vec{f} \in \mathcal{M} ^{\infty,q} (\mathbb{R}^{nd}) \;\; | \;\;
 V_{\vec{\varphi}} \vec{f} \in L^0  (\mathbb{R}^{2nd})\}, \;\; 1 \leq q < \infty, \\
\mathcal{M} ^{p,0} (\mathbb{R}^{nd}) & = & \{ \vec{f} \in \mathcal{M} ^{p,\infty} (\mathbb{R}^{nd}) \;\; | \;\;
 V_{\vec{\varphi}} \vec{f} \in L^0  (\mathbb{R}^{2nd})\}, \;\; 1 \leq p < \infty, \\
\mathcal{M} ^{0,0} (\mathbb{R}^{nd}) & = & \{ \vec{f} \in \mathcal{M} ^{\infty,\infty} (\mathbb{R}^{nd}) \;\; | \;\;
 V_{\vec{\varphi}} \vec{f} \in L^0  (\mathbb{R}^{2nd})\},
\end{eqnarray*}
equipped with the norms of  $ \mathcal{M} ^{\infty,q},  \mathcal{M} ^{p,\infty}$ and $\mathcal{M} ^{\infty,\infty} $ respectively.
Then,
\begin{itemize}
\item[a)] $ \mathcal{M} ^{0,q} $ is $\mathcal{M} ^{\infty,q}-$closure of $ \mathcal{S} $ in $\mathcal{M} ^{\infty,q}$, hence is a closed subspace of
$\mathcal{M} ^{\infty,q}$. Likewise for $\mathcal{M} ^{p,0}$ and $\mathcal{M} ^{0,0}.$
\item[b)] The following duality results hold for $ 1\leq p ,q < \infty$:
$ (\mathcal{M} ^{0,q})' = \mathcal{M} ^{1,q'}, $
$ (\mathcal{M} ^{p,0})' = \mathcal{M} ^{p',1}, $ and $ (\mathcal{M} ^{0,0})' = \mathcal{M} ^{1,1}.$
\end{itemize}
\end{lemma}

From now on we will use these duality relations in the cases $ p  = \infty $ and/or $  q = \infty $
without further explanations.

\par

For the results on multiplication and convolution in modulation spaces and in weighted Lebesgue spaces we first
introduce the \emph{Young functional}:
\begin{equation} \label{R-functional}
\masfR (\mathrm{p}) = \masfR (p_0,p_1,p_2) \equiv 2-\frac 1{p_0}-\frac 1{p_1}-\frac 1{p_2},\qquad \mathrm{p}=
(p_0,p_1,p_2)\in [1,\infty]^3.
\end{equation}

When $ \masfR (\mathrm{p}) = 0, $ the Young inequality for convolution reads as
$$
\| f_1 * f_2 \|_{L^{p_0 ' }} \leq \| f_1  \|_{L^{p_1 }} \| f_2 \|_{L^{p_2}}, \;\;\;
f_j \in L^{p_j }(\mathbb{R}^d), \;\; j = 1,2.
$$

%

The following theorem is an extension of the Young inequality to the case of weighted Lebesgue spaces and
modulation spaces when $ 0\le \masfR (\mathrm{p}) \le 1/2$.

\begin{thm} \label{mainconvolution}
Let $s_j,t_j \in \mathbb R$, $p_j,q_j \in [1,\infty] $, $j=0,1,2$.
Assume that $0\le \masfR (\mathrm{p}) \le 1/2$, $\masfR (\mathrm{q})\le 1$,
\begin{eqnarray}
0&\leq t_j+t_k,   & j,k=0,1,2,  \quad j\neq k, \label{lastineq2A}
\\
0 &\leq t_0+t_1 + t_2 - d \cdot  \masfR (\mathrm{p}), & \text{and} \label{lastineq2B}
\\
0 & \le s_0+s_1+s_2, & \label{lastineq2C}
\end{eqnarray}
with strict inequality in \eqref{lastineq2B} when $\masfR (\mathrm{p})>0$ and $t_j=d\cdot
\masfR (\mathrm{p}) $ for some $j=0,1,2$.

Then  $(f_1,f_2)\mapsto
f_1*f_2$ on $C_0^\infty (\mathbb{R}^d)$ extends uniquely to a continuous
map from
\begin{enumerate}
\item[{\rm{1)}}]  $ L^{p_1} _{t_1}(\mathbb{R}^d) \times  L^{p_2} _{t_2}(\mathbb{R}^d)$ to
$L^{p_0'} _{-t_0}(\mathbb{R}^d)$;

\item[{\rm{2)}}]
$M^{p_1,q_1} _{s_1,t_1}(\mathbb{R}^d) \times  M^{p_2,q_2} _{s_2,t_2}(\mathbb{R}^d)$ to
$M^{p_0',q_0'} _{-s_0,-t_0}(\mathbb{R}^d)$.
\end{enumerate}
\end{thm}

For the proof we refer to \cite{TJPT2014}.
It is based on the detailed study of an auxiliary three-linear map
over carefully chosen regions in $ \mathbb{R}^d $ (see Subsections 3.1 and 3.2 in \cite{TJPT2014}).
This result extends multiplication and convolution properties obtained in \cite{PTT2}.
Moreover, the sufficient conditions from Theorem \ref{mainconvolution} are also necessary in the following sense.

\begin{thm}\label{otpmimality1}
Let $p_j,q_j\in [1,\infty ]$ and $s_j,t_j\in \mathbb R$, $j=0,1,2$. Assume that at
least one of the following statements hold true:
\begin{enumerate}
\item[{\rm{1)}}] the map $(f_1,f_2)\mapsto f_1*f_2$ on $C_0^\infty (\mathbb{R}^d)$ is continuously
extendable to a map from $L^{p_1}_{t_1}(\mathbb{R}^d)\times
L^{p_2}_{t_2}(\mathbb{R}^d)$ to $L^{p_0'}_{-t_0}(\mathbb{R}^d)$;

\vrum

\item[{\rm{2)}}] the map $(f_1,f_2)\mapsto f_1*f_2$ on $C_0^\infty (\mathbb{R}^d)$ is continuously
extendable to a map from $M^{p_1,q_1}_{s_1,t_1}(\mathbb{R}^d)\times
M^{p_2,q_2}_{s_2,t_2}(\mathbb{R}^d)$ to $M^{p_0',q_0'}
_{-s_0,-t_0}(\mathbb{R}^d)$;

\end{enumerate}
Then \eqref{lastineq2A} and \eqref{lastineq2B} hold true.
\end{thm}

\section{Multilinear localization operators} \label{sec02}

In this section we introduce multilinear localization operators in Definition \ref{locopdef}
and show that they can be interpreted as particular Weyl pseudodifferential operators, Theorem \ref{Weyl connection lemma}.
We also introduce multilinear Weyl pseudodifferential operators and prove their connection to the multilinear Wigner transform in
Lemma \ref{L-and-W}. This is done in the context of the duality between $ {\mathcal S}^{( 1)} (\mathbb{R}^d)$ and
$ {\mathcal S}^{( 1)'} (\mathbb{R}^d)$, and carried out verbatim to the duality between $ {\mathcal S} (\mathbb{R}^d)$ and
$ {\mathcal S}' (\mathbb{R}^d)$ in the next Section.

\par

The localization operator $A_a ^{\varphi_1, \varphi_2} $ with the symbol
$a  \in   L^2 (\mathbb{R}^{2d} )$ and with windows $\varphi _1, \varphi _2 \in  L^2 (\mathbb{R}^d ) $
can be defined in terms  of the short-time Fourier transform \eqref{STFT!} as follows:
$$
A_a ^{\varphi_1, \varphi_2}  f(t)=\int_{\mathbb{R}^{2d} } a (x,\omega ) V_{\varphi _1} f (x,\omega ) M_\omega T_x \varphi _2 (t)
\, dx d\omega, \;\;\; f \in   L^2 (\mathbb{R}^d ).
$$

\par

To define multilinear localization operators we slightly abuse the notation (as it is done in e.g. \cite{MOPf}) so that
$ \vec{f} $  will denote both the vector $ \vec{f} = (f_1,f_2, \dots, f_n) $ and the tensor product
$ \vec{f} = f_1 \otimes f_2 \otimes \dots \otimes f_n $. This will not cause confusion, since
the meaning of $ \vec{f} $  will be clear from the context.

For example, if $t=(t_1, t_2, \dots, t_n) $,  and $F_j = F_j (t_j),$ $t_j \in \mathbb{R}^d $,   $ j = 1,2,\dots,n$,
then
\begin{equation} \label{product}
\prod_{j=1} ^{n} F_j (t_j)= F_1 (t_1) \cdot  F_2 (t_2) \cdot \dots \cdot F_n (t_n) = F_1 (t_1) \otimes  F_2 (t_2) \otimes \dots \otimes F_n (t_n) = \vec{F} (t).
\end{equation}

\begin{defn} \label{locopdef}
Let $f_j \in  {\mathcal S}^{( 1)} (\mathbb{R}^d)$, $j=1,2,\dots,n$,
and $ \vec{f} = (f_1,f_2, \dots, f_n) $.
The {\em multilinear localization operator} $\aaf $ with {\em symbol}
$a  \in  {\mathcal S} ^{(1)'} (\mathbb{R}^{2nd} )$ and {\em window}
$$\varphi = (\vec{\varphi}, \vec{\phi})
= (\varphi _1, \varphi _2, \dots, \varphi _n, \phi_1, \phi_2, \dots, \phi _n), \;\;\; \varphi _j, \phi_j \in {\mathcal S} ^{(1)} (\mathbb{R}^d ), \;\;\; j = 1,2,\dots,n,
$$
is given by
\begin{equation} \label{bilinearoperator}
\aaf \vec{f} (t) =\int_{\mathbb{R}^{2nd} } a (x,\omega )
 \prod_{j=1} ^{n}\left ( V_{\varphi _j} f_j (x_j,\omega_j )M_{\omega_j} T_{x_j} \phi_j (t_j) \right )
\, dx d\omega,
\end{equation}
where
$x_j, \omega_j , t_j  \in \mathbb{R}^d, $ $ j=1,2,\dots, n,$ and
$ x = (x_1,x_2, \dots, x_n), $ $ \omega = (\omega_1 ,\omega_2 \dots, \omega_n ), $ $t = (t_1,t_2 \dots, t_n) $.
\end{defn}

\par

\begin{rem}
\label{difference} When $ n = 2 $ in Definition \ref{locopdef} we obtain the bilinear localization operators studied in \cite{Teof2018}. (There is a typo in
\cite[Definition 1]{Teof2018}; the integration in (9) should be taken over $ \mathbb{R}^{4d}$.)

Let $ \mathcal{R}$ denote the trace mapping that assigns to each function $F$
defined on  $ \mathbb{R}^{nd}$ a function defined on $ \mathbb{R}^{d}$ by the formula
$$
 \mathcal{R}: F \mapsto F \left | _{{t_1=t_2=\dots = t_n}} , \;\;\; t_j \in \mathbb{R}^{d}, \; j=1,2,\dots, n. \right .
$$
Then $  \mathcal{R} \aaf $ is the multilinear operator given in \cite[Definition 2.2]{CKasso}.
\end{rem}

\par

By \eqref{STFTtensor-n} it follows that the weak definition of \eqref{bilinearoperator} is given by
\begin{equation} \label{weakblo}
\langle \aaf \vec{f}, \vec{g} \rangle  =
\langle  a V_{\vec{\varphi}} \vec{f} ,
V_{\vec{\phi}} \vec{g}   \rangle
 =
 \langle a, \overline{V_{\vec{\varphi}} \vec{f}} \, V_{\vec{\phi}} \vec{g}  \rangle,
\end{equation}
and $ f_j,g_j,\in  {\mathcal S} ^{(1)} (\mathbb{R}^d )$, $ j=1,2,\dots, n$.
The brackets can be interpreted as  duality between a suitable pair of dual spaces. Thus
$ \aaf $ is  well-defined continuous operator from
$ {\mathcal S} ^{(1)} (\mathbb{R}^{nd} )$ to $({\mathcal S} ^{(1)})' (\mathbb{R}^{2nd} ). $

\par

Next we introduce a class  of multilinear
Weyl pseudodifferential operators ($\Psi$DO for short) and use the  Wigner transform to prove
appropriate interpetation of multilinear localization operators as  multilinear
Weyl pseudodifferential operators, Theorem \ref{Weyl connection lemma}.

Recall that in \cite{CKasso} multilinear localization operators are introduced in connetion to Kohn-Nirenberg $\Psi$DOs instead.

\par

By analogy with the bilinear Weyl pseudodifferential operators given in \cite{Teof2018} we define
the multilinear Weyl pseudodifferential operator  as follows:
\begin{equation} \label{oscilatory-2nd}
L_\sigma( \vec{f} ) (x) = \int_{\mathbb{R}^{2nd}} \sigma (\frac{x+y}{2}, \omega) \vec{f}(y)
e^{2\pi i \mathcal{I}(x-y)\cdot \omega} dy d\omega, \;\;\;
x  \in \mathbb{R}^{nd},
\end{equation}
where $\sigma \in {\mathcal S}^{( 1)'} (\mathbb{R}^{2nd})$,
$ \vec{f} (y) = \prod_{j=1} ^n f_j (y_j)$,
$f_j \in {\mathcal S}^{( 1)} (\mathbb{R}^{d})$, $ j=1,2,\dots, n$.
Here $\mathcal{I}$ denotes the identity matrix in $nd$, that is
$\displaystyle \mathcal{I}(x-y)\cdot \omega = \sum_{j=1} ^n  (x_j - y_j)  \omega_j$.)

\par

Similarly, the  bilinear Wigner transform from \cite{Teof2018} extends to
\begin{multline} \label{cross-Wigner-4d}
W(\vec{f},\vec{g})(x,\omega)  =
\int_{\mathbb{R}^{nd}} \prod_{j=1} ^n \left ( f_j
(x_j +\frac{t_j}2) \overline{g_j(x_j-\frac{t_j}2)} \right )
e^{-2\pi   i \mathcal{I} \omega t}\,dt,
\end{multline}
where $ f_j, g_j \in {\mathcal S}^{( 1)} (\mathbb{R}^{d})$, $x_j, \omega_j , t_j  \in \mathbb{R}^d, $ $ j=1,2,\dots, n,$ and
$ x = (x_1,x_2, \dots, x_n), $ $ \omega = (\omega_1 ,\omega_2 \dots, \omega_n ), $ $t = (t_1,t_2 \dots, t_n) $.

It is easy to see that $ W(\vec{f},\vec{g}) \in  {\mathcal S}^{( 1)} (\mathbb{R}^{2nd})$, when $\vec{f},\vec{g} \in
{\mathcal S}^{( 1)} (\mathbb{R}^{nd})$.

\begin{lemma} \label{L-and-W}
Let $\sigma \in {\mathcal S}^{( 1)} (\mathbb{R}^{2nd})$ and  $ f_j, g_j \in {\mathcal S}^{( 1)} (\mathbb{R}^{d})$, $ j=1,2,\dots, n$.
Then $L_\sigma $ given by \eqref{oscilatory-2nd} extends to a continuous map from ${\mathcal S}^{( 1)} (\mathbb{R}^{nd})$ to ${\mathcal S}^{( 1)'} (\mathbb{R}^{2nd})$ and the following formula holds:
$$
\langle L_\sigma \vec{f} , \vec{g} \rangle = \langle \sigma, W(\vec{g},\vec{f}) \rangle.
$$
\end{lemma}

\begin{proof} The proof follows by the straightforward calculation:
\begin{multline*}
\langle \sigma, W(\vec{g},\vec{f}) \rangle =   \int_{\mathbb{R}^{2nd}}  \sigma (x,\omega) W(\vec{f},\vec{g})(x,\omega) dx d\omega
\\
= \int_{\mathbb{R}^{3nd}} \sigma (x,\omega)
 \prod_{j=1} ^n \left ( f_j
(x_j +\frac{t_j}2) \overline{g_j(x_j-\frac{t_j}2)} \right )
e^{-2\pi   i \mathcal{I} \omega t}
dt  dx d\omega \\
= \int_{\mathbb{R}^{6d}} \sigma (\frac{u+v}{2},\omega)
 \prod_{j=1} ^n \left ( f_j
(v_j) \overline{g_j(u_j)} \right )
e^{-2\pi   i \mathcal{I}(u-v) \omega}
du dv d\omega \\
= \langle \sigma (\frac{u+v}{2},\omega)
\vec{f}(v) e^{2\pi   i \mathcal{I}(u-v) \omega},
\vec{g}(u)
\rangle
=
\langle L_\sigma \vec{f} , \vec{g} \rangle,
\end{multline*}
where we used $W(\vec{g},\vec{f}) = \overline{W(\vec{f},\vec{g})}$ and the change of variables $ u = x+ \frac{t}2,$
$ v= x - \frac{t}2.$ This extends to each $\sigma \in {\mathcal S}^{( 1)'} (\mathbb{R}^{2nd})$,
since  $ W(\vec{f},\vec{g}) \in  {\mathcal S}^{( 1)} (\mathbb{R}^{2nd})$ when
$ f_j, g_j \in {\mathcal S}^{( 1)} (\mathbb{R}^{d})$, $ j=1,2,\dots, n$.

\par
\end{proof}

The so called Weyl connection between the set of
linear localization operators and Weyl $\Psi$DOs  is   well known, we refer to
e.g. \cite{folland89, BCG02, Teof2016}.
The corresponding  Weyl connection in bilinear case
is established in \cite[Theorem 4]{Teof2018}. The proof is quite technical and based on the kernel theorem for Gelfand-Shilov spaces (see e.g. \cite{Prang, TKNN2012, Teof2015}) and direct calculations.
Since the proof of the following Theorem \ref{Weyl connection lemma}  is its straightforward extension, here we only sketch the main ideas.
The conclusion of Theorem \ref{Weyl connection lemma} is that any multilinear
localization operator can be viewed as a particular multilinear Weyl  $\Psi$DOs, as expected.

\par

\begin{thm} \label{Weyl connection lemma}
Let there be given $ a \in  {\mathcal S}^{( 1)'} (\mathbb{R}^{2d})$ and
let $\varphi = (\vec{\varphi}, \vec{\phi}),$
$\vec{\varphi} = (\varphi _1, \varphi _2, \dots, \varphi _n),$
$\vec{\phi} =  ( \phi_1, \phi_2, \dots, \phi _n),$
$\varphi _j, \phi_j \in {\mathcal S} ^{(1)} (\mathbb{R}^d ),$ $ j = 1,2,\dots,n.$
Then the localization operator $\aaf $ is the Weyl pseudodifferential operator with the Weyl symbol
$$ \sigma = a\ast W(\vec{\phi},\vec{\varphi}) = a\ast (\prod_{j=1} ^n  W(\phi_j,\varphi_j)). $$
Therefore, if $ \vec{f} = (f_1,f_2, \dots, f_n),$
$  \vec{g} = (g_1,g_2, \dots, g_n),$
$ f_j, g_j,  \in {\mathcal S}^{( 1)'} (\mathbb{R}^{d}),$   $ j=1,2,\dots, n$,
then
$$
\langle \aaf \vec{f},  \vec{g} \rangle
= \langle L_{a\ast  W(\vec{\phi},\vec{\varphi})} \vec{f},  \vec{g} \rangle.
$$
\end{thm}

\begin{proof}
The formal expressions given below are justified due to the absolute convergence of the involved integrals and
the standard interpretation of oscillatory integrals in distributional setting. We refer to \cite[Section 5]{Teof2018}
for this and for a detailed calculations.

The calculations from the proof of \cite[Theorem 4]{Teof2018}
yield the following  kernel representation of \eqref{weakblo}:
$$
\langle \aaf \vec{f}, \vec{g} \rangle
=
\langle k, \prod_{j=1} ^n \overline{f_j}   \otimes \prod_{j=1} ^n g_j  \rangle,
$$
where the kernel $k = k(t,s)$, is given by
\begin{equation} \label{kernel}
k(t,s)
= \int_{\mathbb{R}^{2nd}} a(x,\omega)
\prod_{j=1} ^n \overline{M_{\omega_j} T_{x_j} \varphi _j} (t) \cdot
\prod_{j=1} ^n M_{\omega_j} T_{x_j} \phi_j  (s) dx d\omega,
\end{equation}
$t = (t_1, t_2,\dots, t_n),$  $s = (s_1, s_2,\dots, s_n )$, $ t_j, s_j \in \mathbb{R}^{d},$ $ j=1,2,\dots, n$.

\par

To calculate  the convolution $ a\ast ( \prod_{j=1} ^n W(\phi_j,\varphi_j)) = a \ast W(\vec{\phi},\vec{\varphi})$
we use  $ W (g,f)  = \overline{W(f,g)}$, the commutation relation $ T_x M_\omega = e^{-2 \pi i x \cdot \omega} M_\omega T_x,$
and the covariance property of the Wigner transform:
$$
W(T_{x_j} M_{\omega_j} \phi_j, T_{x_j} M_{\omega_j}  \varphi_j)  (p_j,q_j)
= W (\phi_j,\varphi_j) (p_j-x_j, q_j-\omega_j), \;\;\; j=1,2,\dots, n.
$$
Let $ p = (p_1,p_2,\dots, p_n),$ $ q = (q_1,q_2\dots, q_n),$ $ p_j,q_j \in \mathbb{R}^{d},$ $j=1,2,\dots, n.$
Then,
\begin{multline} \label{convolution!}
a \ast W(\vec{\phi},\vec{\varphi}) (p,q)
=
\int_{\mathbb{R}^{2nd}}  a(x,\omega) \times \\
\left ( \int_{\mathbb{R}^{nd}} \int_{\mathbb{R}^{nd}}
 \prod_{j=1} ^n
M_{\omega_j}  T_{x_j} \phi_j(p_j+\frac{t_j}{2})
\cdot
 \prod_{j=1} ^n \overline{ M_{\omega_j}  T_{x_j} \varphi_j}
(p_j-\frac{t_j}{2})
e^{-2\pi i q \cdot t}
dt \right) dx d\omega,
\end{multline}
where $q\cdot t $ is the scalar product of $q,t\in  \mathbb{R}^{d},$ cf. \cite[Section 5]{Teof2018}.

\par

Therefore,
\begin{multline*}
\langle L_{ a\ast  W(\vec{\phi},\vec{\varphi})} \vec{f}, \vec{g} \rangle
=
\langle  a\ast  \prod_{j=1} ^n W(\phi_j,\varphi_j), W( \vec{g}, \vec{f}) \rangle
=
\int_{\mathbb{R}^{2nd}}   a(x,\omega) \times
\\[1ex]
\int_{\mathbb{R}^{nd}} \big (
\int_{\mathbb{R}^{nd}}
 \prod_{j=1} ^n
M_{\omega_j}  T_{x_j} \phi_j(p_j+\frac{t_j}{2}) \cdot
\prod_{j=1} ^n \overline{  M_{\omega_j}  T_{x_j} \varphi_j}
(p_j-\frac{t_j}{2})
\times \\
\prod_{j=1} ^n f_j (p_j -\frac{t_j}{2}) \cdot  \prod_{j=1} ^n \overline{g_j}(p_j+\frac{t_j}{2})
dt \big ) dp dx d\omega,
\end{multline*}
Finally, after performing the change of variables we obtain
$$
\langle L_{ a\ast   W(\vec{\phi},\vec{\varphi}) } \vec{f}, \vec{g} \rangle =
\langle k,  \prod_{j=1} ^n \overline{f_j }   \otimes \prod_{j=1} ^n g_j   \rangle,
$$
where the kernel $k$ is given by \eqref{kernel}. The theorem now follows from the uniqueness of the kernel representation, Theorem \ref{kernelteorema}.
\end{proof}

\section{Continuity properties of localization operators} \label{sec03}

We first recall the sharp estimates of the modulation space norm for
the cross-Wigner distribution given in \cite{CN}.
There it is shown that the sufficient conditions for the continuity of the cross-Wigner distribution on
modulation spaces are also necessary (in the un-weighted case).
Related results can be found elsewhere, e.g.  in \cite{Toft2, To8, Teof2016}. In many situations such results overlap.
For example, Proposition 10 in \cite{Teof2018} coincides with certain sufficient conditions from \cite[Theorem 1.1]{CN}
when restricted to $ \masfR (\mathrm{p}) = 0,$ $t_0 = -t_1,$ and $ t_2 = |t_0|$.

\begin{thm} \label{thm-cross-Wigner}
Let there be given $ s \in \mathbb{R}$  and $ p_i, q_i, p, q \in [1,\infty],$ such that
\begin{equation} \label{uslov1}
p \leq p_i, q_i \leq q, \;\;\; i = 1,2
\end{equation}
and
\begin{equation} \label{uslov2}
\min \left \{ \frac{1}{p_1} + \frac{1}{p_2},   \frac{1}{q_1} + \frac{1}{q_2} \right \}
\geq
\frac{1}{p} + \frac{1}{q}.
\end{equation}
If $f,g \in \mathcal{S} (\mathbb{R}^{d}), $
then the map $ (f,g) \mapsto  W(f,g) $
where $ W $ is the cross-Wigner distribution given by \eqref{cross-Wigner}
extends to sesquilinear continuous map from
$ M^{p_1, q_1} _{|s|} (\mathbb{R}^{d}) \times  M^{p_2,q_2} _{s} (\mathbb{R}^{d})  $ to
$ {M}^{p,q} _{s,0} (\mathbb{R}^{2d}) $ and
\begin{equation} \label{cross-Wig-mod}
\|  W(f,g) \|_{ {M}^{p,q} _{s,0} } \lesssim \| f\|_{M^{p_1, q_1} _{|s|}}  \|g \|_{M^{p_2,q_2} _{s} }.
\end{equation}
Viceversa, if there exists a constant $C>0$ such that
$$
\|  W(f,g) \|_{ {M}^{p,q}} \lesssim \| f\|_{M^{p_1, q_1} }  \|g \|_{M^{p_2,q_2} }.
$$
then \eqref{uslov1} and \eqref{uslov2} must hold.
\end{thm}

\begin{proof}
We omit the proof which is given in \cite[Section 3]{CN}, and recall here only the main formulas which highlight its most
important parts.

The first formula is the well-known relation between the Wigner transform and the STFT (see \cite[Lemma 4.3.1]{book}):
$$
W (f,g) (x, \omega) = 2^d e^{4\pi i x \cdot \omega} V_{\overline{g^*}} f (2x, 2\omega), \;\;\;
f,g \in  {\mathcal S} (\mathbb{R}^d).
$$

To estimate the modulation space norm of $ W (f,g) (x, \omega) $ we fix $ \psi_1, \psi_2 \in \mathcal{S} (\mathbb{R}^{d}) \setminus 0 $ and use the fact that
modulation spaces are independent on the choice of the window function
from $\mathcal{S} (\mathbb{R}^{2d}) \setminus 0$, Theorem \ref{modproerties} 1).
By choosing the window to be $W(\psi_1, \psi_2)$, after some calculations we obtain:
\begin{multline*}
(V_{W(\psi_1, \psi_2)} W(g,f)) (z,\zeta)
\\[1ex]
=
e^{-2\pi i z_2 \zeta_2}
\overline{ V_{\psi_1} f} (z_1 + \frac{\zeta_2}{2}, z_2 - \frac{\zeta_1}{2})
V_{\psi_2}g (z_1 - \frac{\zeta_2}{2}, z_2 + \frac{\zeta_1}{2}),
\end{multline*}
cf.  the proof of \cite[Lemma 14.5.1 (b)]{book}. Consequently (cf. \cite[Section 3]{CN}),
\begin{eqnarray*}
\|  W(g,f)  \|_{M^{p,q} _{s,0}}
& \asymp  &
\left( \int_{\mathbb{R}^{2d}}
(| V_{\psi_1} f |^p * |V_{\psi_2} g ^* |^p)^{q/p} (\zeta_2,- \zeta_1)
\langle (\zeta_2,- \zeta_1) \rangle ^{s q} d \zeta \right )^{1/q} \\
& = &
\|  | V_{\psi_1} f |^p * |V_{\psi_2} g ^* |^p\|_{L^{q/p} _{ps,0}}.
\end{eqnarray*}
Then one proceeds with a careful case study to obtain \eqref{cross-Wig-mod}. We refer to \cite{CN} for details.
\end{proof}

From the inspection of the proof of Theorem \ref{thm-cross-Wigner}
given in \cite[Section 3]{CN}, the definition of
$W (\vec{f}, \vec{g}) $ given by \eqref{cross-Wigner-4d} and the use of the kernel theorem
we conclude the following.

\begin{cor} \label{cor-cross-Wigner}
Let the assumptions of Theorem \ref{thm-cross-Wigner} hold.
If $ \vec{f} = (f_1,f_2, \dots, f_n) $, $ \vec{g} = (g_1,g_2, \dots, g_n) $
and $f_j, g_j \in \mathcal{S} (\mathbb{R}^{d}), $ $ j=1,2,\dots, n$,
then the map $ (\vec{f},\vec{g}) \mapsto  W(\vec{f},\vec{g}) $,
where $ W $ is the cross-Wigner distribution given by \eqref{cross-Wigner-4d}
extends to a continuous map from
$\mathcal{M}^{p_1, q_1} _{|s|} (\mathbb{R}^{d}) \times \mathcal{M}^{p_2,q_2} _{s} (\mathbb{R}^{d})  $ to
$ \mathcal{M}^{p,q} _{s,0} (\mathbb{R}^{2d}) $, where the modulation spaces are given by \eqref{modnormtensor}.
\end{cor}

\par

Next we give an  extension of \cite[Theorem 14.5.2]{book} and \cite[Theorem 14]{Teof2018} to the multilinear Weyl $\Psi$DOs.
Recall, if $\sigma \in M^{\infty, 1} (\mathbb{R}^{2d}) $ is the Weyl symbol of $ L_\sigma $, then
\cite[Theorem 14.5.2]{book} says that
$L_\sigma $ is bounded on
$ M^{p,q} (\mathbb{R}^{d})$, $ 1\leq p,q \leq \infty $.
This result has a long history starting from the Calderon-Vaillancourt theorem on boundedness of
the pseudodifferential operators with smooth and bounded symbols on $ L^2 (\mathbb{R}^{d})$, \cite{CalVai}.
It is generalized by Sj{\"o}strand in \cite{Sjostrand1}
where $M^{\infty,1}$ is used as appropriate symbol class.
Sj{\"o}strand's results were thereafter extended in
\cite{Grochenig0,book,Grochenig1b,Toft2,To8,Toft2007a}.
Moreover, we refer to \cite{BGHKasso, BKasso2004,  BKasso2006}
for the multilinear Kohn-Nirenberg $\Psi$DOs, and the recent contribution \cite{CNT}
related to $\tau-$$\Psi$DOs (these include both  Kohn-Nirenberg (when $\tau = 0$) and Weyl operators
(when $\tau = 1/2$).

\par

The following fact related to symbols  $\sigma \in M^{\infty, 1} (\mathbb{R}^{2nd}) $  is a straightforward extension of \cite[Theorem 14]{Teof2018}.

\begin{thm} \label{bounded Weyl}
Let $ \sigma \in M^{\infty, 1} (\mathbb{R}^{2nd}) $ and let   $L_\sigma $ be given by \eqref{oscilatory-2nd}.
The operator $L_\sigma $ is bounded from $ \mathcal{M} ^{p,q} (\mathbb{R}^{nd}) $ to $ \mathcal{M} ^{p,q} (\mathbb{R}^{nd}) $,
$ 1\leq p ,q \leq  \infty,$ with a uniform estimate $ \| L_\sigma \|_{op} \leq \| \sigma \|_{ M^{\infty, 1}} $ for the operator norm.
\end{thm}

On the other hand, Theorem \ref{bounded Weyl} is a special case of \cite[Theorem 5.1.]{CN} if $L_\sigma $ is a linear operator. Here below we give the multilinear version of  \cite[Theorem 5.1.]{CN}.

\begin{thm} \label{bounded Weyl-extended}
Let there be given $ s \geq 0$  and $ p_i, q_i, r_i,  p, q \in [1,\infty],$ such that
\begin{equation} \label{uslov11}
q\leq  \min \{ p_1 ', q_1 ', p_2, q_2 \}
\end{equation}
and
\begin{equation} \label{uslov22}
\min \left \{ \frac{1}{p_1} + \frac{1}{p_2 '},   \frac{1}{q_1} + \frac{1}{q_2 '}  \right \}
\geq
\frac{1}{p'} + \frac{1}{q'}.
\end{equation}
Then the operator $L_\sigma $ given by \eqref{oscilatory-2nd} with symbol
$ \sigma \in M^{p,q} _{s,0} (\mathbb{R}^{2nd}) $, from
$ \mathcal{S} (\mathbb{R}^{nd}) $ to $ \mathcal{S}' (\mathbb{R}^{nd}), $
extends uniquely to a bounded operator from
$ \mathcal{M} ^{p_1,q_1} _{s,0}  (\mathbb{R}^{nd}) $ to $ \mathcal{M} ^{p_2,q_2} _{s,0}  (\mathbb{R}^{nd}) $,
with the estimate
\begin{equation}   \label{zakljucak}
\| L_\sigma \vec{f} \|_{\mathcal{M} ^{p_2,q_2} _{s,0} } \lesssim \| \sigma \|_{  M^{p,q} _{s,0}} \| \vec{f} \|_{\mathcal{M} ^{p_1,q_1} _{s,0}}.
\end{equation}

In particular, when $\sigma \in  M^{\infty, 1}  (\mathbb{R}^{2nd})$ we have $ \| L_\sigma \|_{op} \leq \| \sigma \|_{ M^{\infty, 1}} $ for the operator norm.

Vice versa, if \eqref{zakljucak} holds for $s=0$, and for every
$ \vec{f} \in \mathcal{S} (\mathbb{R}^{nd}) $,
$\sigma \in \mathcal{S}' (\mathbb{R}^{2nd})$, then \eqref{uslov1} and \eqref{uslov2} must be satisfied.
\end{thm}

\par

\begin{proof} The proof is a straightforward extension of the proof of \cite[Theorem 5.1.]{CN}, and we give it here for the sake of completeness.

When $ \vec{f} \in \mathcal{M} ^{p_1,q_1} _{s,0}  (\mathbb{R}^{nd}) $ and $\vec{g} \in
\mathcal{M} ^{p_2 ',q_2 ' } _{s,0}  (\mathbb{R}^{nd})$, their Wigner transform
$W(\vec{g},\vec{f}) = \overline{W(\vec{f},\vec{g})}$
belongs to $  M^{p',q'} _{-s,0} $ since the conditions \eqref{uslov1} and \eqref{uslov2} of Theorem \ref{thm-cross-Wigner}
are transferred to \eqref{uslov11} and \eqref{uslov22}, respectively.

Now, Lemma \ref{L-and-W} and the duality of modulation spaces give
\begin{multline*}
|\langle L_\sigma \vec{f} , \vec{g} \rangle| = |\langle \sigma, W(\vec{g},\vec{f}) \rangle|
\leq  \|  \sigma \|_{ M^{p,q} _{s,0}} \| W(\vec{f},\vec{g}) \|_{M^{p',q'} _{-s,0}} \\
\leq  C \| \vec{f} \|_{ \mathcal{M} ^{p_1,q_1} _{s,0} }
\| \vec{g} \|_{\mathcal{M} ^{p_2 ',q_2 ' } _{s,0}},
\end{multline*}
for some constant $C>0$ (and we used the fact that modulation spaces are closed under the complex conjugation).

We refer to \cite[Theorem 1.1.]{CTW} for the necessity of conditions \eqref{uslov11} and \eqref{uslov22}
(in linear case).
\end{proof}

Next, we combine different results established so far to obtain an extension of \cite[Theorem 15]{Teof2018}.
More precisely, we
use the relation between the Weyl pseudodifferential operators and the localization operators
(Lemma \ref{Weyl connection lemma}),  the convolution estimates for modulation spaces
(Theorem \ref{mainconvolution}), and boundedness of pseudodifferential operators (Theorem \ref{bounded Weyl-extended}) to obtain
continuity results for $ \aaf $ for different choices of windows and symbols.

\par

\begin{thm} \label{conv-a-cross-Wigner}
Let there be given $ s \geq 0$  and $ p_i, q_i, p, q \in [1,\infty],$ $i=0,1,2$ such that
\eqref{uslov11} and \eqref{uslov22} hold.
Moreover, let $q_0 \leq q,$ and
\begin{equation} \label{p_0 uslov}
p_0 \geq p \;\;\; \text{if} \;\;\; p\geq 2, \;\;\; \text{and } \;\;\; \frac{2p}{2-p} \geq p_0 \geq p
\;\;\; \text{if} \;\;\;  2> p \geq 1.
\end{equation}
If $\vec{\varphi} \in \mathcal{M}^{r_1} _{2s,0} (\mathbb{R}^{nd}), $ $\vec{\phi} \in \mathcal{M}^{r_2} _{2s,0} (\mathbb{R}^{nd}), $
where $  \frac{1}{r_1} + \frac{1}{r_2 } \geq 1,$
and $ a \in M^{p_0, q_0} _{s_0,t_0} ( \mathbb{R}^{2nd})$ with $ s_0 \geq -s $,
and $ \displaystyle t_0 \geq d \left ( \frac{1}{p} - \frac{1}{p_0 }  \right ) $ with
the strict inequality  when $ p_0 =p $,
then $ \aaf $ is continuous
from $ \mathcal{M}^{p_1,q_1} _{s,0} (\mathbb{R}^{nd}) $ to  $ \mathcal{M}^{p_2,q_2} _{s,0}(\mathbb{R}^{nd}) $
with
$$
\| \aaf \|_{op} \lesssim  \|a\|_{M^{p_0, q_0} _{s_0,t_0}} \| \vec{\varphi} \|_{\mathcal{M}^{r_1} _{2s,0}}
\| \vec{\phi} \|_ {\mathcal{M}^{r_2} _{2s,0}}.
$$
\end{thm}

\begin{proof}
We first estimate $W (\vec{\phi} ,\vec{\varphi} )$. If  $\vec{\varphi} \in \mathcal{M}^{r_1} _{2s,0} (\mathbb{R}^{nd}), $ $\vec{\phi} \in \mathcal{M}^{r_2} _{2s,0} (\mathbb{R}^{nd}) $, with
 $  \frac{1}{r_1} + \frac{1}{r_2 } \geq 1$, then
Corollary  \ref{cor-cross-Wigner} implies that
$$
W (\vec{\phi} ,\vec{\varphi} )  \in \mathcal{M}^{1, \infty} _{2s,0} (\mathbb{R}^{2nd}).
$$
\par

Now, we use the calculation of  $a \ast W (\vec{\phi} ,\vec{\varphi} )$ from the proof of Theorem \ref{Weyl connection lemma}
(see \eqref{convolution!}) and Theorem \ref{mainconvolution}. The Young functional \eqref{R-functional} becomes
$ \masfR (\mathrm{p}) = \masfR (p', p_0, 1)$, and the condition $ \masfR (\mathrm{p})  \in [0, 1/2]$ is equivalent to
\eqref{p_0 uslov}, while
$ \masfR (\mathrm{q}) = \masfR (q', q_0, \infty) \leq 1 $ is equivalent to $q_0 \leq q$.
Furthermore, \eqref{lastineq2C} transfers to  $ s_0 \geq - s,$
while \eqref{lastineq2A} and \eqref{lastineq2B}  are equivalent to
$ \displaystyle t_0 \geq d \left ( \frac{1}{p} - \frac{1}{p_0 }  \right ) $ with
the strict inequality  when $ p_0 =p $. Therefore, by Theorem  \ref{mainconvolution} 2) we obtain
$$
a \ast  W(\vec{\phi} ,\vec{\varphi}) \in
M^{p_0, q_0} _{s_0,t_0} (\mathbb{R}^{2nd}) \ast  \mathcal{M}^{1, \infty} _{2s,0} (\mathbb{R}^{2nd}) \subset
M^{ p, q} _{s,0} (\mathbb{R}^{2nd}).
$$

\par

Finally, by Theorem \ref{bounded Weyl}  with $\sigma = a \ast   W(\vec{\phi} ,\vec{\varphi}) $, it follows that
$$
\| \aaf \|_{op} =  \| L_\sigma \|_{op} \leq \| \sigma \|_{ M^{ p, q} _{s,0}}
\leq
\|a\|_{M^{p_0, q_0} _{s_0,t_0}} \| \vec{\varphi} \|_{\mathcal{M}^{r_1} _{2s,0} }
\| \vec{\phi} \|_ {\mathcal{M}^{r_2} _{2s,0} },
$$
and the Theorem is proved.
\end{proof}

\par

In particular, we recover (the linear case treated in) \cite[Theorem 5.2]{CN} when $ r_1 = r_2 = r$, $t_0 = 0,$ $ s_0 = -s$, $p_0 = p $ (that is $ \masfR (p',p_0,1) = 0$), and $q_0 = q $ (that is $ \masfR (q', q_0, \infty) = 1$). Therefore, by
\cite[Remark 5.3]{CN},  we obtain an extension of \cite[Theorem 3.2]{CorGro2003} and
\cite[Theorem 4.11]{To8} for this particular choice of weights.

\par

Note that conditions $ \masfR (p',p_0,1) \in (0, 1/2]$ which extends the possible choices of the Lebesgue parameters beyond the usual Young condition  $ \masfR (p',p_0,1) = 0$ must be compensated by an additional condition to the weights, expressed by
$ \displaystyle t_0 \geq d \left ( \frac{1}{p} - \frac{1}{p_0 }  \right ) $.

Another result concerning the boundedness of (bilinear) localization operators on un-weighted modulation spaces is given by
\cite[Theorem 15]{Teof2018}. There we used different type of estimates, leading to the result which partially overlap
with Theorem \ref{conv-a-cross-Wigner}. For example, both results give the same continuity property when the symbol $a$ belongs to
$ a \in M^{\infty, 1}  ( \mathbb{R}^{2nd})$.

%
%
%
%
%
%

\end{document}